\documentclass[a4paper,11pt,twoside]{amsart}
\usepackage[english]{babel}
\usepackage[utf8]{inputenc}

\usepackage[a4paper,inner=3cm,outer=3cm,top=4cm,bottom=4cm,pdftex]{geometry}
\usepackage{fancyhdr}
\usepackage{color}
\usepackage{bold-extra}
\usepackage{ mathrsfs }

\usepackage{comment}
\usepackage{graphics}
\usepackage{aliascnt}
\usepackage[pdftex,citecolor=green,linkcolor=red]{hyperref}

\usepackage{amsmath}
\usepackage{amsfonts}
\usepackage{amssymb}
\usepackage{amsthm}
\usepackage{comment}
\usepackage{mathtools}
\usepackage{enumerate}
\usepackage{enumitem} 

\newtheorem{theorem}{Theorem}[section]
\newtheorem{corollary}[theorem]{Corollary}

\newtheorem{proposition}[theorem]{Proposition}
\theoremstyle{definition}
\newtheorem{definition}[theorem]{Definition}

\numberwithin{equation}{section}

\newcommand\eps{\varepsilon}

\newcommand\F{\mathbb{F}}

\newcommand\R{\mathbb{R}}
\newcommand\Z{\mathbb{Z}}

\def\ord{\operatorname{ord}}

\begin{document}
\title{Upper bound for dimension of Hilbert cubes contained in the quadratic residues of $\F_p$}

\author[Alsetri]{Ali Alsetri}
\address{Department of Mathematics, University of Kentucky\\
715 Patterson Office Tower\\
Lexington, KY 40506\\
USA}
\email{alialsetri@uky.edu}

\author[Shao]{Xuancheng Shao}
\address{Department of Mathematics, University of Kentucky\\
715 Patterson Office Tower\\
Lexington, KY 40506\\
USA}
\email{xuancheng.shao@uky.edu}
\thanks{XS was supported by NSF grant DMS-1802224.}


\maketitle

\begin{abstract}
We consider the problem of bounding the dimension of Hilbert cubes in a finite field $\F_p$ that does not contain any primitive roots. We show that the dimension of such Hilbert cubes is $O_{\eps}(p^{1/8+\eps})$ for any $\eps > 0$, matching what can be deduced from the classical Burgess estimate in the special case when the Hilbert cube is an arithmetic progression. We also consider the dual problem of bounding the dimension of multiplicative Hilbert cubes avoiding an interval.
\end{abstract}

\section{Introduction}

A central theme in additive combinatorics is to study the interplay between arithmetic and multiplicative structures. Let $\F_p$ be a finite field with $p$ prime. In this paper, we investigate the distribution of primitive roots in $\F_p$ in Hilbert cubes.

\begin{definition}[Hilbert cubes]
Let $d$ be  a positive integer. A Hilbert cube $H \subset \F_p$ of dimension $d$ is a set of the form
$$ H = \mathcal{H}(a_0;a_1,\cdots,a_d) := \left\{ a_0 + n_1a_1 + \cdots + n_d a_d \colon n_1,\cdots,n_d \in \{0,1\}\right\} $$
for some $a_0,a_1,\cdots,a_d \in \F_p$, with $a_1,\cdots,a_d$ pairwise distinct.
\end{definition}

Alternatively, $H = \mathcal{H}(a_0;a_1,\cdots,a_d)$ is the collection of all subset sums of $A = \{a_1,\cdots,a_d\}$ translated by $a_0$. If all these subset sums are distinct, then $|H| = 2^d$. In the other extreme, if $A$ is a homogeneous arithmetic progression of the form $A = \{k, 2k, \cdots, dk\}$ for some positive integer $k$, then $|H| \leq d(d+1)/2+1$.

We study the quantity $F(p)$, defined to be the largest positive integer $d$, such that there exists a Hilbert cube of dimension $d$ not containing any primitive roots modulo $p$.

\begin{theorem}\label{thm:F(p)}
For any $\eps > 0$ we have $F(p) \ll_{\eps} p^{1/8+\eps}$.
\end{theorem}

Previously the best known upper bound is $F(p) \ll p^{3/19+o(1)}$ from~\cite{DES17}, building on earlier works~\cite{HS99, DES13}. As noted in~\cite{DES13}, Theorem~\ref{thm:F(p)} implies as a special case the Burgess bound $g(p) \leq p^{1/4+o(1)}$ on the least primitive root $g(p)$ modulo $p$~\cite{Bur62}. This can be seen by considering $H = \mathcal{H}(0; 1, 2, \cdots, d) = \{0, 1, 2, \cdots, d(d+1)/2\}$. Hence any improvement of the exponent $1/8$ in Theorem~\ref{thm:F(p)} would also lead to an improvement of the Burgess bound.

As immediate corollaries, we get the same upper bound for the dimension of Hilbert cubes whose elements are all quadratic residues, or Hilbert cubes whose elements are all quadratic non-residues.

\begin{corollary}\label{cor:F(p)}
If $H \subset \F_p$ is a Hilbert cube of dimension $d$ whose elements are all quadratic residues modulo $p$, then $d \ll_{\eps} p^{1/8+\eps}$ for any $\eps > 0$. Similarly, If $H \subset \F_p$ is a Hilbert cube of dimension $d$ whose elements are all quadratic non-residues modulo $p$, then $d \ll_{\eps} p^{1/8+\eps}$ for any $\eps > 0$.
\end{corollary}

\begin{proof}
If all elements in $H$ are quadratic residues, then $H$ does not contain any primitive roots, and hence the conclusion follows from Theorem~\ref{thm:F(p)}. If all elements in $H = \mathcal{H}(a_0; a_1,\cdots,a_d)$ are quadratic non-residues, then for any fixed quadratic non-residue $g \in \F_p$, all elements in the dilated Hilbert cube $H' = \mathcal{H}(ga_0; ga_1, \cdots, ga_d)$ are all quadratic residues, and the conclusion follows from the previous case.
\end{proof}

Our proof of Theorem~\ref{thm:F(p)} follows the general strategy in~\cite{DES17}. First we locate a large generalized arithmetic progression (GAP) $P$ in the Hilbert cube $H = \mathcal{H}(a_0; a_1,\cdots, a_d)$ when $d \geq p^{1/8+\eps}$; see Proposition~\ref{prop:cube-GAP} below. This is to be expected since there are lots of collisions when forming subset sums of $\{a_1,\cdots,a_d\}$ when $d \geq p^{1/8+\eps}$, and thus $H$ should have rich additive structures. This type of phenomenon from subset sums or iterated sumsets was studied in~\cite{SV06a, SV06b}. Then we use character sum estimates to show that $P$ must contain primitive roots; see Proposition~\ref{prop:char-sum-GAP} below.

We remark that Theorem~\ref{thm:F(p)} explores the interaction between an additively defined set (Hilbert cube) and a multiplicatively defined set (primitive roots), belonging to the broader theme of sum-product phenomenon in additive combinatorics. See~\cite{DES13, DES13b, DES17} for other distributional problems involving the quadratic residues, and~\cite{DE1, DE2} for related questions in the setting of integers instead of $\mathbb{F}_p$.

We also investigate the following dual problem, where the roles of addition and multiplication are reversed. We start with the definition of multiplicative Hilbert cubes. 

\begin{definition}[Multiplicative Hilbert cubes]
Let $d$ be  a positive integer. A multiplicative Hilbert cube $H \subset \F_p$ of dimension $d$ is a set of the form
$$ H = \mathcal{H}^{\times}(a_0;a_1,\cdots,a_d) := \left\{ a_0 a_1^{n_1} \cdots a_d^{n_d} \colon n_1,\cdots,n_d \in \{0,1\}\right\} $$
for some $a_0,a_1,\cdots,a_d \in \F_p^{\times}$, with $a_1,\cdots,a_d$ pairwise distinct.
\end{definition}

\begin{theorem}\label{thm:mult}
For any $\eps > 0$, there exists $\delta > 0$ such that the following statement holds. Let $I \subset \F_p$ be an interval of length $p^{1-\delta}$, and let $H$ be a multiplicative Hilbert cube of dimension $d$ that does not intersect $I$. Then $d \ll_{\eps} p^{\eps}$.
\end{theorem}

Note that if $H \subset \F_p^{\times}$ is a multiplicative subgroup, then it is a multiplicative Hilbert cube of dimension $|H|$. Our proof of Theorem~\ref{thm:mult} uses Bourgain's multilinear exponential sum estimate~\cite{Bourgain} and the Erd\"{o}s-Tur\'{a}n inequality on equidistribution.

\subsection*{Notation}

We use $X \ll Y$, $X = O(Y)$, or $Y \gg X$ to denote the estimate $|X| \leq CY$ for some constant $C$.  If we wish to permit this constant to depend on one or more parameters we shall indicate this by appropriate subscripts, thus for instance $O_{\eps}(Y)$ denotes a quantity bounded in magnitude by $C_{\eps} Y$ for some quantity $C_{\eps}$ depending only on $\eps$. 

If $x$ is a real number, we write $e(x) \coloneqq e^{2\pi i x}$. If $n$ is an element in a finite field $F_p$, we write $e_p(n) \coloneqq e(n/p) = e^{2\pi in/p}$.

\section{Preliminaries}

\subsection{GAPs and sumsets}

For a subset $A \subset \F_p$ and a positive integer $\ell$, we define the $\ell$-fold sumset $\ell A$ to be the set of all sums $a_1+\cdots +a_{\ell}$ with each $a_i \in A$, and define the restricted $\ell$-fold sumset $\ell^*A$ to be the set of all sums $a_1+\cdots+a_{\ell}$ with distinct $a_1,\cdots,a_{\ell} \in A$. We denote by $S_A$ the collection of all elements which can be represented as a sum of distinct members of $A$. Thus $S_A = \cup_{1 \leq \ell \leq |A|} \ell^*A$.

\begin{definition}
A Generalized Arithmetic Progression (GAP) of rank $r$ is a subset $P \subset \F_p$ of the form 
$$ P = \{a_0 + n_1a_1 + \dots + n_ra_r : 0 \leq n_i < N_i\} $$ 
for some positive integers $N_1,\cdots,N_r$ and some $a_0, a_1,\cdots, a_r \in \F_p$.
It is said to be proper if $|P| = \prod_{i=1}^rN_i$.
\end{definition}

The following theorem of Szemeredi and Vu~\cite[Theorem 10.5]{SV06b} allows us to locate a large GAP inside an interated sumset $\ell^*A$.

\begin{theorem}\label{thm:SV}
For any fixed positive integer $r$ there are positive constants $C$ and $c$ depending on $r$ such that the following holds. Let $p$ be a prime, let $A \subset \F_p$ be a subset, and let $\ell \leq |A|/2$ be a positive integer such that $\ell^{r+1}|A| \geq Cp$. Then $\ell^*A$ either contains all of $\F_p$ or contains a proper GAP of rank $r'$ and size at least $c\ell^{r'}|A|$, for some integer $1 \leq r' \leq r$. 
\end{theorem}

\subsection{Character sum estimates}

We collect character sum estimates over GAPs which will be used in the proof of Theorem~\ref{thm:F(p)}. These can be viewed as generalizations of the classical Burgess estimate for character sums. Chang's estimate~\cite{Cha08} gives non-trivial bound for character sums over GAPs of size $p^{2/5+\eps}$.

\begin{theorem}\label{Chang}
Let $p$ be prime, and let $P \subset \F_p$ be a proper GAP of rank $r$ with $|P| > p^{\frac{2}{5} + \eps}$ for some $\eps > 0$. Let $\chi\pmod{p}$ be a non-trivial Dirichlet character. Then
$$ \sum_{n \in P}\chi(n) \ll_{\eps,r} p^{-c}|P| $$ 
for some constant $c = c(\eps, r) > 0$.
\end{theorem}

For GAPs of rank $r=2$, one can improve the threshold $p^{2/5+\eps}$ in Theorem~\ref{Chang} to $p^{1/3+\eps}$, using the following character sum estimates over unions of intervals~\cite[Corollary 1.2]{Sha15} (which builds on works in~\cite{Heath-Brown}).

\begin{theorem}\label{Shao}
Let $p$ be prime and $\eps > 0$. Let $\chi\pmod{p}$ be a non-trivial Dirichlet character. Let $A \subset [1,p]$ be a union of $s$ disjoint intervals $I_1,\dots,I_s$ each of which has length at least $p^{\eps}$. Suppose that $|A|s^{-\frac{1}{2}} > p^{\frac{1}{4} + \eps}$. Then  
$$\sum_{n \in A}\chi(n) \ll_{\eps} p^{-c} |A| $$
for some constant $c = c(\eps) > 0$.
\end{theorem}

\begin{corollary}\label{Chang2}
Let $p$ be prime, and let $P \subset \F_p$ be a proper GAP of rank $2$ with $|P| > p^{\frac{1}{3} + \eps}$ for some $\eps > 0$. Let $\chi\pmod{p}$ be a non-trivial Dirichlet character. Then
$$ \sum_{n \in P}\chi(n) \ll_{\eps} p^{-c}|P| $$ 
for some constant $c = c(\eps) > 0$.
\end{corollary}

\begin{proof}
We may write
$$ P = \{a_0 + n_1a_1 + n_2a_2 \colon 0 \leq n_1 < N_1, 0 \leq n_2 < N_2\} $$
for some positive integers $N_1, N_2$ and some $a_0,a_1,a_2 \in \F_p$.
Without loss of generality, we may assume that $N_1 \geq N_2$. We have $a_1 \neq 0$ since $P$ is proper.
Let
$$ P' = a_1^{-1}P := \{a_1^{-1}a_0 + n_1 + n_2 (a_1^{-1}a_2) \colon  0 \leq n_1 < N_1, 0 \leq n_2 < N_2\}. $$
Since $P$ is proper, $P'$ is also proper and thus it is a disjoint union of $N_2$ intervals of length $N_1$. We have
$$ |P'| N_2^{-1/2} = N_1N_2^{1/2} \geq (N_1N_2)^{3/4} = |P|^{3/4} > p^{1/4+\eps/2}.  $$
Hence we may apply Theorem~\ref{Shao} to conclude that
$$ \sum_{n \in P} \chi(n) = \sum_{n \in P'} \chi(n) \ll_{\eps} p^{-c}|P| $$
for  some constant $c = c(\eps) > 0$.
\end{proof}

\subsection{Uniform Distribution}

In the proof of Theorem~\ref{thm:mult}, we will need the Erd\"{o}s-Tur\'{a}n inequality that connects equidistribution with exponential sums; see~\cite[Corollary 1.1]{10lectures}.

\begin{theorem}[Erd\"{o}s-Tur\'{a}n inequality]\label{erdos-turan}
Let $u_1,u_2,\cdots,u_N$ be any sequence of points on the unit circle $\R/\Z$. For any positive integer $K$ and any $\alpha \leq \beta \leq \alpha+1$, we have
$$ \left|\#\{1 \leq n \leq N \colon u_n \in [\alpha,\beta]\pmod{1}\} - (\beta-\alpha)N\right|  \leq \frac{N}{K+1} + 3\sum_{k=1}^{K}\frac{1}{k}\left|\sum_{n=1}^Ne(ku_n)\right|.$$
\end{theorem}

\begin{corollary}\label{erdos-turan-cor}
Let $I \subset \F_p$ be an interval of length $|I| = \delta p$, and let $a_1,a_2,\cdots,a_N$ be any sequence of points in $\F_p$, none of which lies in $I$. Then there exists a positive integer $k \leq 10\delta^{-1}$ such that
$$ \left|\sum_{n=1}^N e_p(ka_n) \right| \gg \delta^2N. $$
\end{corollary}

\begin{proof}
We apply Theorem~\ref{erdos-turan} with the sequence of points $\{a_n/p\}_{1 \leq n \leq N}$ and $[\alpha,\beta] = p^{-1}I$ to obtain
$$ \frac{|I|}{p} N \leq \frac{N}{K+1} + 3\sum_{k=1}^K \frac{1}{k} \left|\sum_{n =1}^N e_p(ka_n)\right|
$$
for any positive integer $K$. Choosing $K = 10/\delta$, we conclude that
$$ \delta N \ll \sum_{k=1}^K \left|\sum_{n =1}^N e_p(ka_n)\right|. $$
The conclusion follows immediately.
\end{proof}

\section{Proof of Theorem~\ref{thm:F(p)}}

In this section we deduce Theorem~\ref{thm:F(p)} by showing that Hilbert cubes must contain large proper GAPs (Proposition~\ref{prop:cube-GAP}), and that large GAPs must contain primitive roots (Proposition~\ref{prop:char-sum-GAP}).

\begin{proposition}[Hilbert cubes contain large GAPs]\label{prop:cube-GAP}
For any positive integer $r$, there exist constants $C,c>0$ depending only on $r$ such that the following statement holds. Let $d$ be a positive integer and $p$ be prime. If  $d^{r+2} \geq Cp$, then any Hilbert cube of dimension $d$ in $\F_p$ either contains all of $\F_p$, or contains a proper progression of rank $r'$ and size at least $cd^{r'+1}$ for some $1 \leq r' \leq r$.
\end{proposition}

\begin{proof}
 Let $H = \mathcal{H}(a_0; a_1,\cdots,a_d)$ be a Hilbert cube of dimension $d$, where $a_0,a_1,\cdots,a_d \in \F_p$ with $a_1,\cdots,a_d$ pairwise distinct. We will apply Theorem~\ref{thm:SV} with $A = \{a_1,\cdots,a_d\}$ and $\ell = \lfloor d/2\rfloor$. We have
$$ \ell^{r+1} |A| \gg_r d^{r+2} \geq Cp. $$
Hence the assumptions in Theorem~\ref{thm:SV} are satisfied provided that $C$ is large enough in terms of $r$. Thus we conclude $\ell^*A$ either contains all of $\F_p$ or contains a proper GAP of rank $r'$ and size at least $\gg_r \ell^{r'}|A| \gg_r d^{r'+1}$, for some $1 \leq r' \leq r$. The desired conclusion follows since $\ell^*A$ is contained in a translate of $H$.
\end{proof}

\begin{proposition}\label{prop:char-sum-GAP}
Let $P \subset \F_p$ be a proper GAP of rank $r$. If $P$ does not contain any primitive roots modulo $p$, then $|P| \ll_{\eps} p^{f(r)+\eps}$ for any $\eps > 0$, where
\[ f(r) = \begin{cases} 1/4 & \text{if }r=1, \\ 1/3 & \text{if }r=2, \\ 2/5 & \text{if }r \geq 3. \end{cases} \]
\end{proposition}

\begin{proof}
We may assume that $p$ is sufficiently large in terms of $\eps$, since otherwise the claim holds trivially.
By~\cite[Lemma 2.4]{DES13} we have for any $n \in \F_p$,
\[ \frac{\varphi(p-1)}{p-1} \sum_{\chi\pmod p} c_{\chi} \chi(n) = \begin{cases} 1 & \text{if }n\text{ is a primitive root}\pmod{p}, \\ 0  & \text{otherwise,} \end{cases} \]
where $c_{\chi} = \mu(\ord(\chi))/\varphi(\ord(\chi))$ and $\ord(\chi)$ denotes the order of $\chi$. Since $P$ does not contain any primitive roots, we have
\[ \sum_{\chi\pmod{p}} c_{\chi}  \sum_{n \in P} \chi(n) = 0. \]
Taking out the term $\chi = \chi_0$, we get
\[ \left|\sum_{\chi\neq\chi_0} c_{\chi} \sum_{n \in P} \chi(n)\right| \geq |P|-1. \]
Note that
$$ \sum_{\chi} |c_{\chi}| \leq  \sum_{d\mid p-1} \frac{1}{\varphi(d)} \#\{\chi\colon \ord(\chi)=d\} = \sum_{d\mid p-1} 1 \ll_{\tau} p^{\tau} $$
for any $\tau > 0$.
Thus for at least one non-trivial character $\chi \neq \chi_0$ we have
$$ \left| \sum_{a \in P} \chi(a) \right| \gg_{\tau} |P|p^{-\tau} $$
for any $\tau > 0$.

Suppose, for the purpose of contradiction, that $|P| \geq p^{f(r)+\eps}$.  By the Burgess estimate on character sums (in the case $r=1$),  Theorem~\ref{Chang} (in the case $r \geq 3$), and Corollary~\ref{Chang2} (in the case $r=2$), we have
$$ \left| \sum_{a \in P} \chi(a) \right| \ll_{\eps} |P|p^{-c}. $$
for some constant $c = c(\eps) > 0$.  This leads to a contradiction, choosing $\tau = c/2$ (say).
\end{proof}

We are now ready to deduce Theorem~\ref{thm:F(p)}.
We may assume that $p$ is sufficiently large in terms of $\eps$, since otherwise the claim holds trivially. Suppose, for the purpose of contradiction, that there is a Hilbert cube $H \subset \F_p$  of dimension $d > p^{1/8+\eps}$ not containing any primitive roots. By Proposition~\ref{prop:cube-GAP} applied with $r=10$ (say), $H$ contains a proper progression $P$ of rank $r'$ and size $|P| \gg d^{r'+1} \gg p^{(r'+1)/8+\eps}$, for some $1 \leq r' \leq 10$. On the other hand, since $P$ does not contain any primitive roots, Proposition~\ref{prop:char-sum-GAP} implies that $|P|\ll_{\eps} p^{f(r')+\eps/2}$. Combining the upper and lower bounds for $|P|$, we conclude that $(r'+1)/8 < f(r')$. This is a contradiction, no matter whether $r'=1$, $r'=2$, or $r' \geq 3$.

\section{Proof of Theorem~\ref{thm:mult}}

In this section we deduce Theorem~\ref{thm:mult} by combining Corollary~\ref{erdos-turan-cor} with certain (weighted) exponential sum estimates over Hilbert cubes.

We may assume that $p$ is sufficiently large, otherwise the claim holds trivially.
Let $H = \mathcal{H}^{\times}(a_0; a_1, \cdots, a_d)$ be a multiplicative Hilbert cube, where $a_0,a_1,\cdots,a_d \in \F_p^{\times}$ with $a_1,\cdots,a_d$ pairwise distinct. Suppose, for the purpose of contradiction, that $d \geq p^{\eps}$. Set $r = \lceil 2\eps^{-1}\rceil$, and form a partition
$$ \{a_1,\cdots,a_d\} = A_1 \cup \cdots \cup A_r $$
into subsets of almost-equal sizes, so that
$$ |A_i| = \frac{d}{r} + O(1) \gg_{\eps} p^{\eps} $$
for each $1 \leq i \leq r$. For $b \in \F_p^{\times}$, consider the exponential sum
$$ S_b := \sum_{x_1 \in A_1,\cdots,x_r \in A_r} e_p(bx_1\cdots x_r). $$
On the one hand, none of the points $a_0x_1\cdots x_r$ with $x_i \in A_i$ lies in the interval $I$, so Corollary~\ref{erdos-turan-cor} implies that there exists a positive integer $k \leq 10p^{\delta}$ such that
$$ \left| \sum_{x_1 \in A_1, \cdots, x_r \in A_r} e_p(ka_0x_1\cdots x_r) \right| \gg p^{-2\delta} |A_1|\cdots |A_r|. $$
On the other hand, since $|A_i| \gg_{\eps} p^{\eps}$ for each $i$ and $\prod_{1 \leq i \leq r}|A_i| \gg_{\eps} p^{r\eps} \gg p^2$, we may apply Bourgain's exponential sum bound~\cite[Theorem A]{Bourgain} to obtain
$$ \left| \sum_{x_1 \in A_1, \cdots, x_r \in A_r} e_p(ka_0x_1\cdots x_r) \right| < p^{-c} |A_1|\cdots |A_r| $$
for some constant $c = c(\eps) > 0$. This leads to a contradiction by choosing $\delta = c/4$.

\bibliographystyle{plain}
\bibliography{biblio}

\end{document}